\documentclass[11pt]{amsart}
\usepackage{amsmath,amsthm,amscd,amsfonts, amssymb}
\usepackage{graphicx}
\usepackage{tikz}
\usepackage{color}
\include{diagxy}
\usepgflibrary{shapes}
\newcommand{\sect}[1]{\section{#1}\setcounter{equation}{0}}

\font\mbn=msbm10 scaled \magstep1
\font\mbs=msbm7 scaled \magstep1
\font\mbss=msbm5 scaled \magstep1
\newfam\mbff
\textfont\mbff=\mbn
\scriptfont\mbff=\mbs
\scriptscriptfont\mbff=\mbss

\newcommand{\N}       { \mathbb{N}}
\newcommand{\Z}        {\mathbb{Z}  }

\newtheorem{Th}{Theorem}[section]
\newtheorem{Lm}[Th]{Lemma}
\newtheorem{C}[Th]{Corollary}

\newtheorem{R}[Th]{Remark}

\begin{document}
\title[Banach spaces of polynomials]{Banach spaces of polynomials as ``large'' subspaces of ${\mathbf \ell^\infty}$-spaces}

\author{Alexander Brudnyi} 
\address{Department of Mathematics and Statistics\newline
\hspace*{1em} University of Calgary\newline
\hspace*{1em} Calgary, Alberta\newline
\hspace*{1em} T2N 1N4}
\email{albru@math.ucalgary.ca}

\thanks{Research supported in part by NSERC}

\begin{abstract}
In this note we study Banach spaces of traces of real polynomials on $\mathbb R^n$ to compact subsets equipped with supremum norms from the point of view of Geometric Functional Analysis.
\end{abstract}

\date{}

\maketitle

\sect{Main Results}
Recall that the {\em Banach-Mazur distance} between two $k$-dimensional real Banach spaces $E,F$ is defined as
\[
d_{BM}(E,F):=\inf\{\|u\|\cdot\|u^{-1}\|\},
\]
where the infimum is taken over all isomorphisms $u:E\rightarrow F$. We say that $E$ and $F$ are equivalent if they are isometrically isomorphic (i.e., $d_{BM}(E,F)=1$).  Then
$\ln d_{BM}$  determines a metric on the set $\mathcal B_k$ of equivalence classes of isometrically isomorphic $k$-dimensional Banach spaces (called the {\em Banach-Mazur compactum}). It is known that $\mathcal B_k$ is compact of $d_{BM}$-``diameter'' $\sim k$, see \cite{G}.

Let $C(K)$ be the Banach space of real continuous functions on a compact Hausdorff space $K$ equipped with the supremum norm.  Let $F\subset C(K)$ be a filtered subalgebra with filtration $\{0\}\subset F_0\subseteq F_1\subseteq\cdots\subseteq F_d\subseteq\cdots\subseteq F$ (that is, $F=\cup_{d\in\Z_+}F_i$ and $F_i\cdot F_j\subset F_{i+j}$ for all $i,j\in\Z_+$) such that $n_d:={\rm dim}\, F_d<\infty$ for all $d$. 
In what follows we assume that $F_0$ contains constant functions on $K$. Our main result is



\begin{Th}\label{te1}
Suppose there are $c\in\mathbb R$ and $\{p_d\}_{d\in\N}\subset\N$ such that
\begin{equation}\label{e2}
\frac{\ln n_{d\cdot p_d}}{p_d}\le c\quad\text{for all}\quad d\in\N.
\end{equation}
Then there exist linear injective maps $i_d: F_d\hookrightarrow \ell_{n_{d\cdot p_d}}^\infty$ such that
\[
d_{BM}(F_d, i_d(F_d))\le e^{c},\quad d\in\mathbb N.
\]
\end{Th}
As a corollary we obtain:
\begin{C}\label{c1}
Suppose $\{n_d\}_{d\in\N}$ grows at most polynomially in $d$, that is, 
\begin{equation}\label{eq1.2}
\exists\, k, \hat c\in\mathbb R_+ \ \text{such that}\ \forall\, d\quad n_d\le \hat c d^k.
\end{equation}
Then for each natural number $s\ge 3$ there exist linear injective maps $i_{d,s}: F_d\hookrightarrow \ell_{N_{d,s}}^\infty$, where $N_{d,s}:= \left\lfloor \hat cd^k\cdot s^k\cdot \left(\lfloor\ln(\hat cd^k)\rfloor +1\right)^k\right\rfloor$,  such that 
\[
d_{BM}(F_d, i_{d,s}(F_d))\le \left(es^k\right)^{\frac 1s},\quad k\in\mathbb N.
\]
\end{C}
Let $\mathcal F_{\hat c,k}$ be the family of all possible filtered algebras $F$ on compact Hausdorff spaces $K$ satisfying condition \eqref{eq1.2}.
By $\mathcal B_{\hat c, k, \bar n_d}\subset\mathcal B_{\bar n_d}$ we denote the closure in $\mathcal B_{\bar n_d}$ of the set formed by all subspaces $F_d$ of algebras $F\in\mathcal F_{\hat c,k}$
having a fixed dimension $\bar n_d\in\mathbb N$.

Corollary \ref{c1} allows to estimate the metric entropy of $\mathcal B_{\hat c, k, \bar n_d}$. Recall that for a compact subset $S\subset  \mathcal B_{\bar n_{d}}$ its {\em $\varepsilon$-entropy} ($\varepsilon>0$) is defined as
$H(S,\varepsilon):=\ln N(S,d_{BM}, 1+\varepsilon)$, 
where $N(S,d_{BM}, 1+\varepsilon)$ is the smallest number of open $d_{BM}$-``balls'' of radius $1+\varepsilon$ that cover $S$.
\begin{C}\label{c2}
For $k\ge 1$ there exists a numerical constant $C$ such that for each $\varepsilon\in (0,\frac{1}{2}]$ 
\[
H\bigl(\mathcal B_{\hat c, k, \bar n_d},\varepsilon\bigr)\le (Ck\cdot\ln(k+1))^{k}\cdot (\hat cd^{k})^2\cdot (\ln (\hat cd^k)+1)^{k+1}\cdot\left(\frac{1}{\varepsilon}\right)^k\cdot \left(\ln\left(\frac{1}{\varepsilon}\right)\right)^{k+1}.
\]
\end{C}

\sect{Basic Example: Banach Spaces of Polynomials}
Let  $\mathcal P_d^n$ be the space of real polynomials on $\mathbb R^n$ of degree at most $d$.  For a compact subset $K\subset\mathbb R^n$ by $\mathcal P_d^n|_K$ we denote the trace space of restrictions of polynomials in  $\mathcal P_d^n$  to $K$ equipped with the supremum norm. Applying Corollary \ref{c1} to algebra $\mathcal P^n|_K:=\cup_{d\ge 0}\,\mathcal P_d^n|_K$ we obtain:\medskip 

\noindent (A) {\em There exist  linear injective maps $i_{d, K}: \mathcal P_d^n|_K\hookrightarrow \ell_{N_{d,n}}^\infty$, where 
\begin{equation}\label{ndn}
N_{d,n}:=\left\lfloor e^{2n}\cdot (n+2)^{2n}\cdot d^n\cdot \left(2n+1+\lfloor n\ln d\rfloor\right)^n\right\rfloor,
\end{equation}
such that}
\begin{equation}\label{bmdn}
d_{BM}\left(\mathcal P_d^n|_K, i_{d,K}( \mathcal P_d^n|_K)\right)\le\left(e\cdot (n+2)^2\right)^{\frac{1}{n+2}}\, (< 2.903).
\end{equation}
\smallskip

Indeed,
\begin{equation}\label{ind}
\widetilde N_{d,n}:={\rm dim}\,\mathcal P_d^n|_K\le {d+n \choose n}<\left(\frac{e\cdot (d+n)}{n}\right)^n\le\left(\frac{e\cdot (1+n)}{n}\right)^n\cdot d^n<e^{2n}\cdot d^n.
\end{equation}
Hence, Corollary \ref{c1} with $c=e^{2n}$, $k:=n$ and $s:=(n+2)^2$ implies the required result.\smallskip

If $K$ is $\mathcal P^n$-{\em determining} (i.e., no nonzero polynomial vanish on $K$), then $\widetilde N_{d,n}={d+n \choose n}$ and so for some constant $c(n)$ (depending on $n$ only) we have
\begin{equation}\label{e3.2}
\widetilde N_{d,n}<N_{d,n}\le c(n)\cdot \widetilde N_{d,n}\cdot\bigl(1+\ln\widetilde N_{d,n}\bigr)^n.
\end{equation}
Hence, $V_{d,n}(K):=i_{d,K}( \mathcal P_d^n|_K)$ is a {\em ``large'' subspace} of $\ell_{N_{d,n}}^\infty$ in terminology of \cite{B}. Therefore from (A) and  \cite[Prop.~8]{B} applied to $V_{d,n}(K)$ we obtain:
\medskip

\noindent (B) {\em There is a constant $c_1(n)$ (depending on $n$ only) such that for each $\mathcal P^n$-determining compact set $K\subset\mathbb R^n$ there exists an $m$-dimensional subspace $F\subset \mathcal P_d^n|_K$ with }
\begin{equation}\label{e3.3}
m:={\rm dim}\,F>c_1(n)\cdot\bigl(\widetilde N_{d,n}\bigr)^{\frac 12}\quad\text{ and}\quad
d_{BM}(F,\ell_m^\infty)\le 3.
\end{equation}

{\em In turn, if $\hat d\in\N$ is such that $N_{\hat d,n}\le c_1(n)\cdot\bigl(\widetilde N_{d,n}\bigr)^{\frac 12}$, then due to property (A) for each $\mathcal P^n$-determining compact set $K'\subset\mathbb R^n$ there exists a $\widetilde N_{\hat d,n}$-dimensional subspace $F_{\hat d,n,K'}\subset F$ such that}
\begin{equation}\label{star}
d_{BM}\bigl(F_{\hat d,n,K'},\mathcal P_{\hat d}^n|_{K'}\bigr)<9.
\end{equation}
\smallskip

Further, the dual space $(V_{d}^n(K))^*$ of $V_{d}^n(K)$ is the quotient space of $\ell_{N_{d,n}}^1$. In particular, the closed ball of $(V_{d}^n(K))^*$
contains at most $c(n)\cdot \widetilde N_{d,n}\cdot\bigl(1+\ln\widetilde N_{d,n}\bigr)^n$ extreme points, see \eqref{e3.2}. Thus the balls of $(V_{d}^n(K))^*$ and $V_{d}^n(K)$ are ``quite different'' as convex bodies. This is also expressed in the following property (similar to the celebrated John ellipsoid theorem but with an extra logarithmic factor) which is a consequence of property (A) and \cite[Prop.~1]{DMTJ}:
\medskip

\noindent (C) {\em There is a constant $c_2(n)$ (depending on $n$ only) such that for all $\mathcal P^n$-determining compact sets $K_1, K_2\subset\mathbb R^n$}
\begin{equation}\label{e3.4}
d_{BM}\bigl(\mathcal P_d^n|_{K_1},(\mathcal P_d^n|_{K_2})^*\bigr)\le c_2(n)\cdot (\widetilde N_{d,n}\cdot (1+\ln \widetilde N_{d,n}))^{\frac 12}.
\end{equation}
\smallskip

A stronger inequality is valid if we replace $(\mathcal P_d^n|_{K_2})^*$ above by $\ell_{\widetilde N_{d,n}}^1$, see \cite[Th.~2]{DMTJ}.
\begin{R}\label{re3.1}
{\rm Property (C) has the following geometric interpretation. By definition, $(\mathcal P_d^n|_{K_2})^*$ is a 
$\widetilde N_{d,n}$-dimensional real Banach space generated by evaluation functionals $\delta_x$ at points $x\in K_2$ with the closed unit ball being the balanced convex hull of the set $\{\delta_x\}_{x\in K_2}$. Thus $K_2$ admits a natural isometric embedding into the unit sphere of  $(\mathcal P_d^n|_{K_2})^*$. Moreover, the Banach space of linear maps $(\mathcal P_d^n|_{K_2})^*\rightarrow \mathcal P_d^n|_{K_1}$ equipped with the operator norm is isometrically isomorphic to the Banach space of real polynomial maps $p:{\mathbb R}^n\rightarrow \mathcal P_d^n|_{K_1}$ of degree at most $d$ (i.e., $f^*\circ p\in\mathcal P_d^n$ for all $f^*\in (\mathcal P_d^n|_{K_1})^*$) with norm $\|p\|:=\sup_{x\in K_2}\|p(x)\|_{\mathcal P_d^n|_{K_1}}$.
Thus property (C) is equivalent to the following one:}\smallskip

\noindent {\rm (C$'$)} There exists a polynomial map $p:{\mathbb R}^n\rightarrow \mathcal P_d^n|_{K_1}$ of degree at most $d$ such that the balanced convex hull of $p(K_2)$ contains the closed unit ball of $\mathcal P_d^n|_{K_1}$ and is contained in the closed ball of radius
$c_2(n)\cdot (\widetilde N_{d,n}\cdot (1+\ln \widetilde N_{d,n}))^{\frac 12}$ of this space (both centered at $0$).
\end{R}

Our next property, a consequence of Corollary \ref{c2} and \eqref{ind}, estimates the metric entropy of the closure of the set $\widetilde{\mathcal P}_{d,n}\subset\mathcal B_{\widetilde N_{d,n}}$ formed by all $\widetilde N_{d,n}$-dimensional spaces $\mathcal P_d^n|_{K}$ with $\mathcal P^n$-determining compact subsets $K\subset\mathbb R^n$.\medskip

\noindent (D) {\em There exists a numerical constant $c>0$ such that for each $\varepsilon\in (0,\frac{1}{2}]$,}
\begin{equation}\label{entro}
H\bigl({\rm cl}(\widetilde{\mathcal P}_{d,n}),\varepsilon\bigr)\le (cn^{2}\cdot\ln(n+1))^{n}\cdot d^{2n}\cdot (1+\ln d)^{n+1}\cdot\left(\frac{1}{\varepsilon}\right)^n\cdot \left(\ln\left(\frac{1}{\varepsilon}\right)\right)^{n+1}.
\end{equation}

\begin{R}\label{rem3.2}
{\rm The above estimate shows that  $\widetilde{\mathcal P}_{d,n}$ with sufficiently large $d$ and $n$ is much less massive than $\mathcal B_{\widetilde N_{d,n}}$. Indeed, as follows from \cite[Th.~2]{Br} 
\[H(\mathcal B_{\widetilde N_{d,n}},\varepsilon)\sim \left(\frac{1}{\varepsilon}\right)^{\frac{\widetilde N_{d,n}-1}{2}}\quad \text{as}\quad \varepsilon\rightarrow 0^+
\]
(here the equivalence depends on $d$ and $n$ as well). On the other hand, \cite[Th.~1.1]{P} implies that for any $\varepsilon>0$,}
\[
0<\liminf_{\widetilde N_{d,n}\rightarrow\infty}\frac{\ln H(\mathcal B_{\widetilde N_{d,n}},\varepsilon)}{\widetilde N_{d,n}}\le \limsup_{\widetilde N_{d,n}\rightarrow\infty}\frac{\ln H(\mathcal B_{\widetilde N_{d,n}},\varepsilon)}{\widetilde N_{d,n}} <\infty.
\]

{\rm It might be of interest to} find sharp asymptotics of $H\bigl({\rm cl}(\widetilde{\mathcal P}_{d,n}),\varepsilon\bigr)$ as $\varepsilon\rightarrow 0^+$ and $d\rightarrow\infty$, {\rm and to} compute (up to a constant depending on $n$) $d_{BM}$-``diameter'' of $\widetilde{\mathcal P}_{d,n}$.
\end{R}

Similar results are valid for $K$ being a compact subset of a real algebraic variety $X\subset\mathbb R^n$ of dimension $m<n$ such that if a polynomial vanishes on $K$, then it vanishes on $X$ as well. In this case there are positive constants $c_X, \tilde c_X$ depending on $X$ only such that $\tilde c_X d^m\le {\rm dim}\,\mathcal P_d^n|_K\le c_X d^m$. For instance, 
Corollary \ref{c1}  with  $c=c_X$, $k:=m$ and $s:=(m+2)^2$ implies that $\mathcal P_d^n|_K$ is linearly embedded into $\ell_{N_{d,X}}^\infty$, where $N_{d,X}:= \left\lfloor c_X d^m\cdot (m+2)^{2m}\cdot \left(\lfloor\ln(c_Xd^m)\rfloor +1\right)^m\right\rfloor$,
with distortion $<2.903$. We leave the details to the reader.\smallskip

\section{Proofs}

\begin{proof}[Proof of Theorem \ref{te1}]
Since ${\rm dim}\, F_i=n_i$, $i\in\N$, and evaluations $\delta_z$ at points $z\in K$ determine bounded linear functionals on $F_i$, the Hahn-Banach theorem implies easily that ${\rm span}\,\{\delta_z\}_{z\in K}= F_i^*$. Moreover, $\|\delta_z\|_{F_i^*}=1$ for all $z\in K$ and the closed unit ball of $F_i^*$ is the balanced convex hull of the set $\{\delta_z\}_{z\in K}$. Let $\{f_{1i},\dots,f_{n_i i}\}\subset F_i$ be an Auerbach basis with the dual basis $\{\delta_{z_{1i}},\dots,\delta_{z_{n_i i}}\}\subset F_i^*$, that is, $f_{ki}(\delta_{z_{li}}):=f_{ki}(z_{li})=\delta_{kl}$ {\rm (the\ Kronecker-delta)} and
$\|f_{ki}\|_K=1$ for all $k$. (Its construction is similar to that of the fundamental Lagrange interpolation polynomials for $F_i=\mathcal P_i^n|_K$, see, e.g., \cite[Prop.~2.2]{BY}.) 
 
Now, we use a ``method of  E.~Landau''  (see, e.g., \cite[Ch.\,3,  \S\,2]{PS}).

By the definition, for each $g\in F_i$ we have $g(z)=\sum_{k=1}^{n_i}f_{ki}(z)g(z_{ki})$, $z\in K$. Hence,
$\|g\|_K\le n_i\|g\|_{\{z_{1i},\dots, z_{n_i i}\}}$. Applying the latter inequality to $g=f^{p_d}$, $f\in F_d$, containing in $F_{i}$, $i:=d\cdot p_d$, and using condition \eqref{e2} we get
for $A_d:=\{z_{1 i},\dots, z_{n_{i} i}\}\subset K$
\[
\|f\|_K=\left(\|g\|_K\right)^{\frac{1}{p_d}}\le \left(n_{d\cdot p_d}\right)^{\frac{1}{p_d}}\cdot \left(\|g\|_{A_d}\right)^{\frac{1}{p_d}}\le
e^{c}\cdot\|f\|_{A_d}.
\]

Thus, restriction $F_d\mapsto F_d|_{A_d}$ determines the required map $i_d: F_d\hookrightarrow\ell_{n_{d\cdot p_d}}^\infty$.   
\end{proof}
\begin{proof}[Proof of Corollary \ref{c1}]
We set $p_d:=s\cdot(\lfloor\ln(\hat cd^k)\rfloor +1)$, $d\in\N$. Then the condition of the corollary implies
\[
\frac{\ln n_{d\cdot p_d}}{p_d}\le \frac{\ln (\hat cd^k)+k\ln p_d}{p_d}\le \frac 1s+\frac{k\ln s}{s}=:c.
\]
Thus the result follows from Theorem \ref{te1}.
\end{proof}

\begin{proof}[Proof of Corollary \ref{c2}]
We make use of \cite[Lm.~1.2]{P} adapted to our setting:
\begin{Lm}\label{le3.2}
Let $S_{\bar n_d}\subset {\mathcal B}_{\bar n_d}$ be the subset formed
by all $\bar n_d$-dimensional subspaces of $\ell_{N_{d,s}}^\infty$. Consider $0 <\xi < \frac{1}{\bar n_d}$ and let $R =\frac{1 + \xi \bar n_d}{1 -\xi \bar n_d}$.
Then $S_{\bar n_d}$ admits an $R$-net $T_R$ of cardinality at most $\left(1 + \frac{2}{\xi}\right)^{N_{d,s}\cdot \bar n_d}$. 
\end{Lm}
Now given $\varepsilon\in (0,\frac 12]$ we choose $s=\lfloor s_\varepsilon\rfloor$ with $s_\varepsilon$ satisfying
$(es_\varepsilon^k)^{\frac{1}{s_\varepsilon}}=\sqrt[4]{1+\varepsilon}$ and $\xi$ such that $R=R_\varepsilon=\sqrt[4]{1+\varepsilon}$. Then according to Corollary \ref{c1} and Lemma \ref{le3.2}, ${\rm dist}_{BM}\bigl(T_{R_{\varepsilon}},\mathcal B_{\hat c, k, \bar n_d}\bigr)< \sqrt{1+\varepsilon}$. 
For each $p\in T_{R_{\varepsilon}}$ we choose $q_p\in \mathcal B_{\hat c, k, \bar n_d}$ such that $d_{BM}(p,q_p)<\sqrt{1+\varepsilon}$. Then the multiplicative triangle inequality for $d_{BM}$ implies that open $d_{BM}$-``balls'' of radius $1+\varepsilon$ centered at points $q_p$, $p\in T_{R_{\varepsilon}}$, cover $\mathcal B_{\hat c, k, \bar n_d}$. Hence,
\begin{equation}\label{e3.6}
N(\mathcal B_{\hat c, k, \bar n_d},d_{BM}, 1+\varepsilon)\le {\rm card}\, T_{R_{\varepsilon}}\le \left(1 + \frac{2}{\xi}\right)^{N_{d,s}\cdot \bar n_d}.
\end{equation}

Next, the function $\varphi(x)=\ln (ex^k)^{\frac{1}{x}}$ decreases for $x\in [e^{\frac{k-1}{k}},\infty)$ and $\lim_{x\rightarrow\infty}\varphi(x)=0$. Its inverse $\varphi^{-1}$ on this interval has domain $(0,e^{-\frac{k-1}{k}}]$, increases and is easily seen (using that $\varphi\circ\varphi^{-1}={\rm id}$) to satisfy
\[
\varphi^{-1}(x)\le \frac{3k}{x}\cdot\ln\left(\frac{3k}{x}\right), \qquad x\in (0,e^{-\frac{k-1}{k}}].
\]
Since $\frac{1}{4}\ln(1+\varepsilon)<e^{-\frac{k-1}{k}}$ for $\varepsilon\in (0,\frac 12]$, the required $s_\varepsilon$ exists and the previous inequality implies that 
\begin{equation}\label{e3.7}
s_\varepsilon \le  \frac{12k}{\ln(1+\varepsilon)}\cdot\ln\left(\frac{12k}{\ln(1+\varepsilon)}\right).
\end{equation}

Further, we have
\begin{equation}\label{e3.8}
\frac{1}{\xi}=\frac{\bar n_d(1+R_\varepsilon)}{R_\varepsilon-1}=\frac{\bar n_d(\sqrt[4]{1+\varepsilon}+1)}{\sqrt[4]{1+\varepsilon}-1}=
\frac{\bar n_d(\sqrt[4]{1+\varepsilon}+1)^2\cdot (\sqrt{1+\varepsilon}+1)}{\varepsilon}.
\end{equation}
From \eqref{e3.6}, \eqref{e3.7}, \eqref{e3.8} invoking the definition of $N_{d,s}$ we obtain 
\[
\ln N(\mathcal B_{\hat c, k, \bar n_d},d_{BM}, 1+\varepsilon)\le  \bar n_d\hat cd^k\left(\ln(\hat cd^k)+1\right)^k\ln\left(\frac{21\bar n_d}{\varepsilon}\right) \left(\frac{12k}{\ln(1+\varepsilon)}\ln\left(\frac{12k}{\ln(1+\varepsilon)}\right)\right)^k.
\]
Using that $\bar n_d\le\hat cd^k$ and the inequality $\frac 23 \cdot\varepsilon\le\ln(1+\varepsilon)$, $\varepsilon\in (0,\frac 12]$, we get the required estimate.
\end{proof}


\end{document}